\documentclass{article}

\usepackage{amssymb,amsthm,amsmath, graphicx,psfrag}
\usepackage{subfigure}
\usepackage{graphics}
\usepackage{epstopdf}
\usepackage{a4}

\newtheorem{thm}{Theorem}[section]
\newtheorem{conj}[thm]{Conjecture}
\newtheorem{cor}[thm]{Corollary}
\newtheorem*{cl}{Claim}
\newtheorem{lem}[thm]{Lemma}
\newtheorem{prop}[thm]{Proposition}

\theoremstyle{remark}

\newdimen\margin   
\def\textno#1&#2\par{%
   \margin=\hsize
   \advance\margin by -4\parindent
          \setbox1=\hbox{\sl#1}%
   \ifdim\wd1 < \margin
      $$\box1\eqno#2$$%
   \else
      \bigbreak
      \hbox to \hsize{\indent$\vcenter{\advance\hsize by -3\parindent
      \sl\noindent#1}\hfil#2$}%
      \bigbreak
   \fi}

\begin{document} 

\title{Embedding cycles of given length in oriented graphs}
\author{Daniela K\"uhn, Deryk Osthus, Diana
Piguet\thanks{The research leading to these results has received funding from the European Union Seventh
Framework Programme (FP7/2007-2013) under grant agreement no. PIEF-GA-2009-253925.}}
\date{}
\maketitle 

\begin{abstract}
Kelly, K\"uhn and Osthus conjectured that for any
$\ell\ge 4$ and the smallest number $k\ge 3$ that does not
divide~$\ell$, any large enough oriented graph~$G$ with
$\delta^+(G), \delta^-(G)\ge \lfloor |V(G)|/k\rfloor +1$
contains a directed cycle of length~$\ell$. We prove this
conjecture asymptotically for the case when~$\ell$ is large
enough compared to $k$ and $k\ge 7$. The case
when~$k\le 6$ was already settled asymptotically by Kelly, K\"uhn and Osthus. 
\end{abstract}

\section{Introduction}

\emph{Oriented graphs} are obtained from undirected graphs by
giving each edge a direction. The \emph{minimum semidegree}
$\delta^0(G)$ of an oriented graph~$G$ is the minimum of its
 \emph{minimum outdegree} $\delta^+(G)$ and of its
 \emph{minimum indegree} $\delta^-(G)$. A \emph{directed} cycle is a cycle in which all the edges are oriented
consistently. Directed paths and walks are defined analogously.
An \emph{$\ell$-cycle} is a directed cycle of length exactly~$\ell$. The girth~$g(G)$ of an oriented
 graph~$G$ is the smallest number~$\ell$ so that~$G$ contains
 an $\ell$-cycle.

 A central problem in this area is the following
 conjecture by Caccetta and
 H\"aggkvist~\cite{cite:Cacc-Hagg}.

\begin{conj}[Caccetta-H\"aggkvist]\label{conj:Cacc-Hagg}
An oriented graph on~$n$ vertices with minimum outdegree~$d$
has girth at most $\lceil \frac
nd\rceil$.
\end{conj}

Despite much work over the years by a large number of
researchers even the case $\lceil\frac nd\rceil=3$
remains open. We refer to~\cite{cite:Sulivan} for a survey on the topic and
to~\cite{cite:HKN-ArXiv} for the currently best bound for the case
$\lceil\frac nd\rceil=3$.  The natural and related question of what minimum semidegree in an
oriented graph forces cycles of length exactly~$\ell\ge 4$ was raised in~\cite{cite:KKO}.

\begin{conj}[Kelly, K\"uhn, Osthus]\label{conj:KDD}
Let $\ell\ge 4$ be an integer and let $k\ge 3$ be
minimal such that~$k$ does not divide~$\ell$. Then there
exists an integer $n_0=n_0(\ell)$ such that every oriented graph~$G$ on $n\ge n_0$ vertices with minimum semidegree
$\delta^0(G)\ge\lfloor n/k\rfloor+1$ contains an $\ell$-cycle.
\end{conj}

Note that $k$ in Conjecture~\ref{conj:KDD} must be of the form $p^s$ for some prime $p$ and some $s\in \mathbb{N}$.
If true Conjecture~\ref{conj:KDD} is tight as a blow-up
of a $k$-cycle has minimum semidegree $\lfloor\frac nk\rfloor$
and there is no $\ell$-cycle (since~$k$ does not divide $\ell$). Proposition~\ref{pro:counter} shows that the condition on $n$ being sufficiently large is necessary.

In~\cite{cite:KKO} Kelly, K\"uhn and Osthus proved Conjecture~\ref{conj:KDD} exactly for
$k=3$ and proved it asymptotically for $k=4$ and $\ell \ge 42$ as well as $k=5$ and $\ell \ge 2550$. 
They also showed that a bound of $\lfloor n/3 \rfloor +1$ suffices for any $\ell \ge 4$.
In
this paper we prove Conjecture~\ref{conj:KDD} asymptotically
for the case when $k\ge 7$ and $\ell$ is sufficiently large compared to $k$.

\begin{thm}\label{thm:cycles}Let $k\ge 7$ and
$\ell\ge 10^7 k^6$. Suppose that~$k$ is the
minimal integer greater than~$2$ that does not divide~$\ell$. Then for
all~$\eta>0$ there exists an integer $n_1=n_1(\eta, \ell)$ such that every
oriented graph~$G$ on $n\ge n_1$ vertices with minimum
semidegree $\delta^0(G)\ge(1+\eta) n/k$ contains an $\ell$-cycle.
\end{thm}

Instead of an $\ell$-cycle, one may want to find a
cycle of length~$\ell$ with some given orientation of its edges. In this case the semidegree condition
seems to depend on the 
so-called cycle-type. Given an arbitrary orientation of a cycle~$C$, the
\emph{cycle-type}~$t(C)$ of~$C$ is the absolute value of the number of edges
with clockwise orientation minus the number of edges with anticlockwise
orientation. So an~$\ell$-cycle has cycle-type~$\ell$ and oriented cycles with cycle-type~$0$
are precisely those for which there is digraph homomorphism
into a directed path. Moreover if $t(C)\ge 3$ then~$t(C)$ is
the maximum length of a directed cycle into which there is a
digraph homomorphism of~$C$. The authors of~\cite{cite:KKO}
made the following conjecture.

\begin{conj}[Kelly, K\"uhn, Osthus]\label{conj:cycle-type}
Let~$C$ be an arbitrarily oriented cycle of length $\ell\ge 4$
and cycle-type $t(C)\ge 4$. Let $k\ge 3$ be minimal such
that~$k$ does not divide~$t(C)$. Then there exists an
integer $n_0=n_0(\ell)$ such that every oriented graph~$G$
on $n\ge n_0$ vertices with minimum semidegree
$\delta^0(G)\ge \lfloor n/k\rfloor+1$ contains~$C$.
\end{conj}

Again the blow-up of a $k$-cycle shows that
Conjecture~\ref{conj:cycle-type} is tight.
As observed in~\cite{cite:KKO}, Conjecture~\ref{conj:KDD}
would imply an approximate version of
Conjecture~\ref{conj:cycle-type}. The same argument shows that
Theorem~\ref{thm:cycles} (together with the cases $k=3,4,5$ of
Conjecture~\ref{conj:KDD} proved in~\cite{cite:KKO}) implies
an approximate version of Conjecture~\ref{conj:cycle-type} for
the case when the
cycle-type~$t(C)$ is sufficiently large with respect to~$k$.
More precisely Theorem~\ref{thm:cycles} implies the following
statement.

\begin{thm}\label{thm:cycle-type} Let $k\ge 3$ and
$t\ge 10^7 k^6$. Suppose that~$k$ is the
minimal integer greater than~$2$ that does not
divide~$t$. Then for all $\eta>0$ and all~$\ell$ there exists an integer $n_1=n_1(\eta, \ell)$ such
that every oriented graph~$G$ on $n\ge n_1$ vertices with minimum
semidegree $\delta^0(G)\ge (1+\eta)n/k$ contains any oriented
cycle~$C$ of length~$\ell$ and cycle-type $t(C)=t$.
\end{thm}

Note that if there exists no cycle of length $\ell$ with cycle-type~$t$ (for example in the case when $\ell<t$), then
Theorem~\ref{thm:cycle-type} is vacuously true.

For completeness we give an outline of the proof of
Theorem~\ref{thm:cycle-type}. Let~$G$ be as in
Theorem~\ref{thm:cycle-type}. Apply a version of Szemer\'edi's
regularity lemma for directed graphs to~$G$ (such a directed
version was proved by Alon and Shapira~\cite{cite:AS}) to
obtain a directed cluster graph~$H'$ with similar minimum semidegree as
in~$G$, i.e.\ $\delta^0(H')\ge (1+\eta/2)|V(H')|/k$.
However~$H'$ needs not to be oriented, but for every double
edge of~$H'$ one can select one of the two edges randomly
(with suitable probability) in order to obtain an oriented
spanning subgraph $H$ of $H'$ which still satisfies
$\delta_0(H)\ge (1+\eta /4)|V(H)|/k$ (see Lemma~3.2
in~\cite{cite:KKO2} for a proof). The oriented graph~$H$
satisfies the conditions of Theorem~\ref{thm:cycles} (with
$\eta/4$ playing the role of $\eta$) and thus contains a $t$-cycle. Now one can 
apply an oriented version of the embedding lemma to find a copy of~$C$ within the subgraph
of~$G$ which corresponds to the $t$-cycle in~$H$. (For the embedding
lemma we refer the reader for example to Lemma 7.5.2.
in~\cite{cite:Diest}.)

We conclude the introduction by raising several open
problems. The first and perhaps most interesting one is whether we can replace the semidegree
condition by an outdegree condition in
Theorems~\ref{thm:cycles} and~\ref{thm:cycle-type}. The
second question is to find the exact
bound for the semidegree condition in more cases. The smallest
open case is for~$\ell=6$, as the case when $\ell\ge 4$ with
$\ell\not \equiv 0 \mod 3$ is solved in~\cite{cite:KKO}. 
Finally, can we prove these results without the use of the
regularity lemma? The proofs of Theorems~\ref{thm:cycles}
and~\ref{thm:cycle-type} rely on a version of the regularity
lemma (in the proof of Theorem~\ref{thm:cycles} the
regularity lemma comes in because it was used
in~\cite{cite:KKO} to prove Lemma~\ref{lem:KKO}). Therefore
the bound on $n_1$ is huge. Proving these theorems without the
help of the regularity lemma would significantly reduce the
bound on~$n_1$. Related problems can be found in the survey~\cite{KOsurvey}.

\section{Proof of Theorem~\ref{thm:cycles}}

Throughout this section the numbers~$k$ and~$\ell$ are
fixed and satisfy the assumptions of Theorem~\ref{thm:cycles}, i.e.

\textno
$k\ge 7$, $\ell \ge 10^7 k^6$ and $k$ is the minimal integer greater than~2 that does not divide $\ell$. &(*)

By the following lemma from~\cite{cite:KKO} it suffices to
find a closed directed walk of length~$\ell$ instead of an $\ell$-cycle. 

\begin{lem}[Kelly, K\"uhn, Osthus]\label{lem:KKO}
Let $\ell \ge 3$ be an integer and~$c>0$. Suppose that there
exists an integer~$n_0$ such that every oriented graph~$H$ on
$n> n_0$ vertices with $\delta^0(H)\ge cn$ contains a  closed
directed walk of length~$\ell$. Then for each
$\varepsilon>0$ there exists $n_1=n_1(\varepsilon, \ell, n_0)$ such that if~$G$ is
an oriented graph on $n\ge n_1$ vertices with $\delta^0(G)\ge
(c+\varepsilon)n$ then~$G$ contains an $\ell$-cycle.
\end{lem}

The proof of Lemma~\ref{lem:KKO} is similar to the argument
showing that Theorem~\ref{thm:cycles} implies
Theorem~\ref{thm:cycle-type}. As there, one first applies
the regularity lemma for directed graphs to~$G$ to obtain a
directed cluster graph~$H'$. The next step is then to find
an oriented cluster 
graph~$H$. As before $\delta^0(H)\ge c|V(H)|$
and so~$H$ contains a closed directed walk of
length~$\ell$, which can then easily be converted to an
$\ell$-cycle in~$G$.

\begin{prop}\label{prop:s-cycle}
Let $k$ and $\ell$ be as in $(*)$.
If~$H$ is an oriented graph of girth less than $k$, then it
contains a closed directed walk of length~$\ell$. 
\end{prop}

\begin{proof}
Let $s<k$ be such that $H$ contains an $s$-cycle~$C$.
Then $(*)$ implies that~$s$ divides~$\ell$.
Thus winding~$\ell/s$ times around the cycle~$C$ gives a
closed directed walk of length~$\ell$.
\end{proof}

By Lemma~\ref{lem:KKO} and Proposition~\ref{prop:s-cycle} we may
assume that our oriented graph~$H$ has girth at least~$k$. Our next aim is to
show that in this case~$H$ contains a directed path of length
at most~$64k$ between any two vertices in~$H$. For this we
use the following result by Shen~\cite{cite:shen}.

\begin{thm}[Shen]\label{thm:Shen}
An oriented graph~$H$ on~$n$ vertices with minimum outdegree
$\delta^+(H)\ge d$ has girth  $$g(H)\leq 3\left\lceil\left(\ln
\frac{2+\sqrt{7}}{3}\right)\frac{n}{d}\right\rceil.$$  
\end{thm}

\begin{cor}\label{cor-shen}
Let $k\in \mathbb N$ with $k\ge 7$. Then an oriented graph~$H$ on~$n$ vertices with minimum outdegree
$\delta^+(H)\ge 63n/32k$ has girth $g(H)<k.$
\end{cor}

\begin{proof}
For $k\ge 10>\frac {10^3}{101}$, Theorem~\ref{thm:Shen} implies that the girth of $H$ is at most 
\[3\left\lceil\left(\ln\frac{2+\sqrt{7}}{3}\right)\frac {32k}{63}\right\rceil\le 3\left( \left (\ln\frac {2+\sqrt{7}}{3}\right)\frac {32k}{63}+1\right)
\le 0.67k+\frac {3\cdot 101k}{10^3}
=0.973k.\]
For the values $k=7,8,9$, from Theorem~\ref{thm:Shen} we obtain that the girth of $H$ is at most $6$.
\end{proof}

\begin{lem}\label{lem-diam} Let $k\in \mathbb N$ with $k\ge 7$ and let~$H$ be an
oriented graph on $n\ge 64k$ vertices, with $\delta^0(H)>n/k$ and
with girth $g(H)\ge k$. Then there is a directed path of length at most~$64k$ from any
vertex~$x$ of~$H$ to any other vertex~$y$ of~$H$.
\end{lem}

\begin{proof}
We claim that any subset~$A$ of order $a\le \frac
n2$ contains a vertex with at least $
n/64k\ge 1$
outneighbours outside~$A$. Suppose not. Then the minimum
outdegree of $H[A]$ is at least $\frac {63n}{64k}\ge \frac
{63a}{32k}$. Now Corollary~\ref{cor-shen} implies that the
girth of~$H[A]$ (and thus of $H$) is less than $k$, a contradiction.

Set $A_1:= N^+(x)\cup \{x\}$. We have $|A_1|> n/k$. If
$|A_1|\le n/2$ then by the above claim there is a vertex
$v_1\in A_1$ with $|N^+(v_1)\setminus A_1| \ge 
n/64k$. In this case set $A_2:= A_1\cup N^+(v_1)$ and
continue in this way. 
After at most~$32k$ steps, we
have $|A_i|> n/2$, for some~$i\le 32k$.

Set $B_1:= N^-(y)\cup \{y\}$ and proceed analogously to obtain
that $|B_j|>n/2$ for some $j\le 32k$. Then $A_i\cap B_j\neq
\emptyset$ and there exists a directed path of length at most $64k$ from $x$ to $y$ whose vertices lie in $A_i\cup B_j$.
\end{proof}

\begin{lem}\label{lem:or-cycles}
Let $k$ and $\ell$ be as in $(*)$.
Let~$H$ be an oriented graph which contains a directed path of length at most~$64k$ 
from any vertex to any other vertex. If~$H$ contains a
closed walk~$W$ with~$a$ edges going forward and~$b$ edges going backwards,
for some $a\neq b$ with $a+b=|V(W)|<k$, then there is a closed directed walk of length~$\ell$ in~$H$.
\end{lem}

\begin{proof}
Let~$W$,~$a$ and~$b$ be as in the lemma.
We may assume without loss of generality  that
\begin{equation}\label{eq:a<b}
a>b.
\end{equation} 

We now consider any maximal subwalk  in~$W$ with all edges
oriented backwards. Let~$x$ and~$y$ denote the endvertices of this walk such that all the edges are oriented from $x$ towards $y$.
Let~$P_{yx}$ be a directed path in~$H$ from~$y$ to~$x$ of length at most~$64 k$. We find such a
path for each maximal subwalk of~$W$ that is oriented
backwards. Let~$P$ be the union of all these
paths~$P_{yx}$ and~$l(P)$ be the sum of their lengths. So
$l(P)<\frac k2\cdot 64k=32k^2$.

There exists a closed directed walk~$W_1$ of length $a+l(P)$
in~$H$ consisting of the edges of~$W$ going forwards and~$P$.
Also there exists a closed directed walk~$W_2$ of length
$a+b+2l(P)$ using each edge of the closed walk~$W$ exactly once and
each path~$P_{yx}$ exactly twice (see Figure~\ref{fig:walks}).
In the case when~$b=0$ we have $W=W_1=W_2$.

\begin{figure}[h]
     \centering
     \subfigure[The walk~$W_1$ uses only edges from the
     walk~$W$ going forwards and the paths~$P_{yx}$
     otherwise.]{
          \label{fig:W1}
          \includegraphics[width=.47\textwidth]{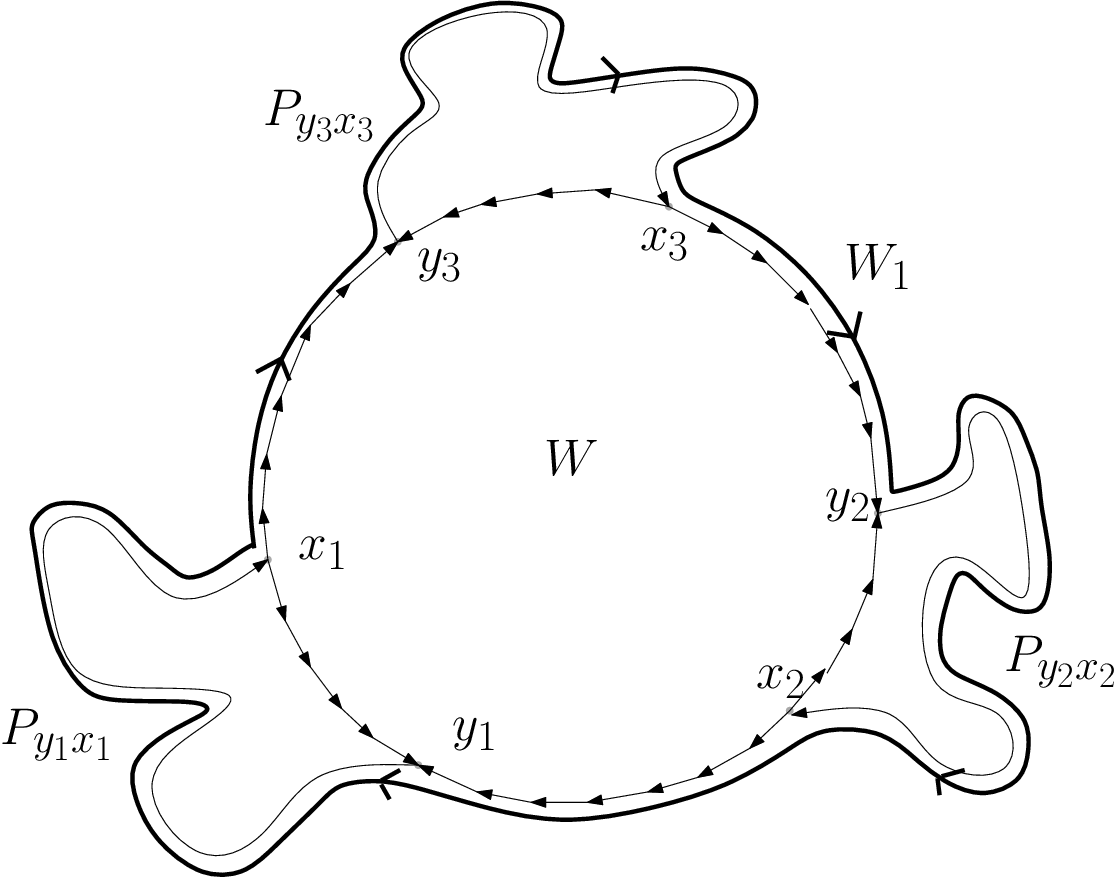}}
     \hspace{.05in}
     \subfigure[The walk~$W_2$ uses every edge of the walk~$W$
     once and each path~$P_{yx}$ twice.]{
          \label{fig:W_2}
          \includegraphics[width=.47\textwidth]{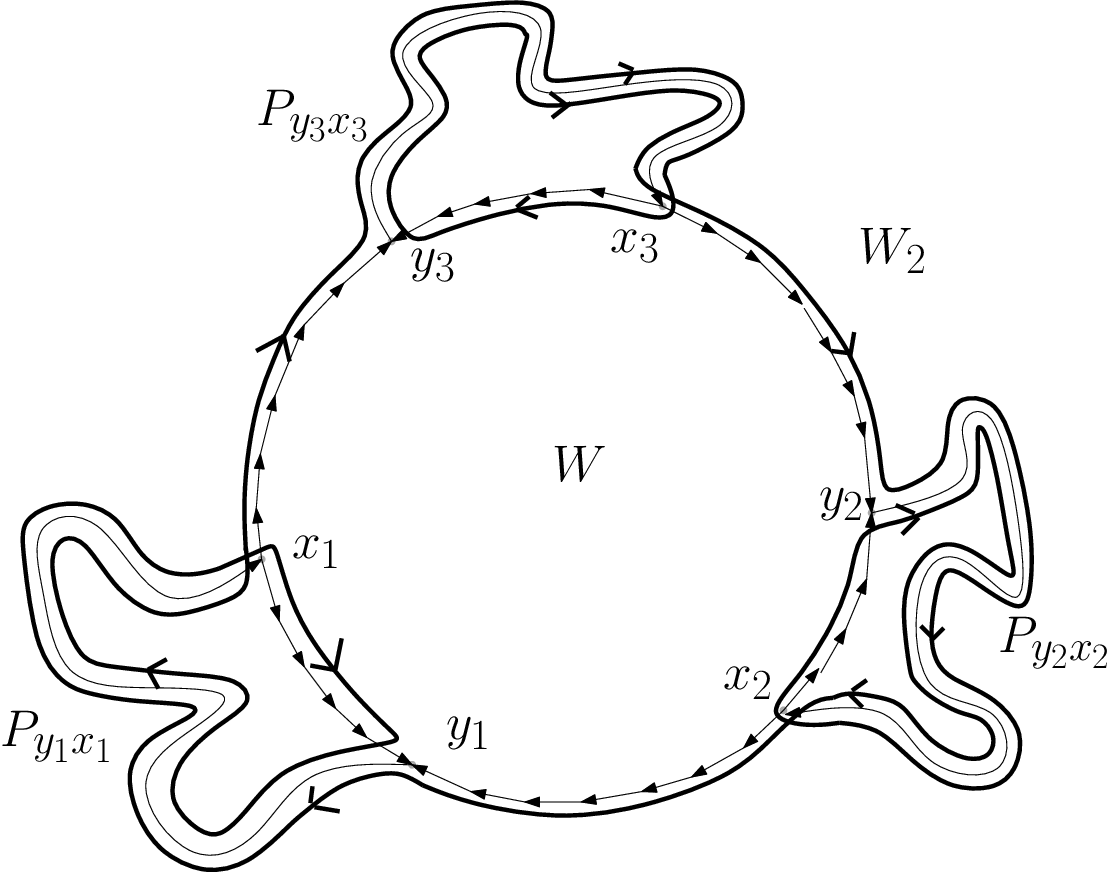}}
          \caption{The two different walks~$W_1$ and~$W_2$.
          For simplicity the figure illustrates the case when the
          closed walk~$W$ is a cycle.}
\label{fig:walks}
\end{figure}

By~\eqref{eq:a<b} we have $2l(W_1)-l(W_2)=a-b>0$, where $l(W_1)$ and $l(W_2)$ denote the lengths of the walks $W_1$ and $W_2$, respectively.
Let~$h:=gcd(2l(W_1),l(W_2)))$ be the greatest common divisor
of~$2l(W_1)$ and~$l(W_2)$. Thus 
\begin{equation}\label{eq:h}
0<h=gcd(a-b, l(W_2))\le a-b<k.
\end{equation}
 By B\'ezout's identity
there are $p,q\in \mathbb Z$ such that
\begin{equation}\label{eq:1}
q\cdot 2l(W_1)+p\cdot l(W_2)=h.
\end{equation}
Moreover we may assume that $0\le q< l(W_2)< k+64k^2$ and $-p< 2l(W_1)< 2k+64k^2$.
Indeed, by replacing $q$ with $q+s\cdot l(W_2)$ and $p$ with $p-s\cdot 2l(W_1)$
for a suitable $s\in\mathbb{Z}$ we may assume that $0\le q<l(W_2)$. But then 
\begin{equation}\label{eq:-p}
-p\overset{\eqref{eq:1}}{=}(q\cdot 2l(W_1)-h)/l(W_2)<2l(W_1).
\end{equation}

From $(*)$ and $(\ref{eq:h})$ we obtain that~$h$ divides~$\ell$. Let $\ell'=\frac \ell h$,
$r=\lfloor\frac {\ell'}{2l(W_1)+l(W_2)}\rfloor$, 
$u=2rh+(\ell'-(2l(W_1)+l(W_2))r)2q$ and 
$v=rh+(\ell'-(2l(W_1)+l(W_2))r)p$.  Let~$W^*$ be the
closed walk obtained by winding around~$W_1$~$u$ times and
winding around~$W_2$~$v$ times. Note that~$W^*$ is well defined. For this, first note that $q \ge 0$ implies that $u \ge 0$.
Secondly, by~\eqref{eq:-p} we also have
\begin{align*}
v &\ge rh-(\ell'-(2l(W_1)+l(W_2))r)(2k+64k^2)\\
&>h\left(\frac
{\ell'}{2l(W_1)+l(W_2)}-1\right)-(2l(W_1)+l(W_2))(2k+64k^2)\\
&\ge h\frac {\ell'}{2l(W_1)+l(W_2)}-h-(3k+128k^2)(2k+64k^2)\\
&\overset{(*)\&\eqref{eq:h}}{>}  \frac {\ell}{129k^2}- k-129k^2\cdot 65k^2\overset{(*)}{\ge} 0.
\end{align*}
The total length of the
closed directed walk~$W^*$ is
\begin{align*}
&u\cdot l(W_1)+v\cdot l(W_2)\\
&=rh(2l(W_1)+l(W_2))+(\ell'-(2l(W_1)+l(W_2))r)(q\cdot
2l(W_1)+p\cdot l(W_2))\\
&\overset{\eqref{eq:1}}{=}rh(2l(W_1)+l(W_2))+(\ell'-(2l(W_1)+l(W_2))r)h\\ &=\ell'h=\ell,
\end{align*}
as required.
\end{proof}

\begin{cor}\label{cor:or-cycle}
Let $k$ and $\ell$ be as in $(*)$. Suppose
that~$H$ is an oriented graph which contains
a directed path of length at most~$64k$ from any vertex to any other vertex.
If for some odd $s<k$ the graph~$H$ contains some
orientation of a cycle of length~$s$, then~$H$ contains a
closed directed walk of length~$\ell$. 
\end{cor}

 Recall that we may assume that our oriented graph~$H$ has girth at least~$k$. Thus
 Lemma~\ref{lem-diam}, Corollary~\ref{cor:or-cycle} and the
following result of Andr\'asfai, Erd\H os and S\'os (see Remark~1.6 in~\cite{cite:AES}) together
imply that we may assume that the underlying graph of~$H$ is
bipartite. 
 
\begin{thm}[Andr\'asfai, Erd\H os, S\'os]\label{thm:AES}
Let $k\ge 4$. Any undirected graph on~$n$ vertices which contains
no odd cycle of length less than~$k$ and has minimum degree
$\delta> \frac {2n}{k}$ is bipartite.
\end{thm}

%

\begin{lem}\label{lem:bip-case}
Let $k\in \mathbb N$  be odd or divisible by~$4$ and let~$H$ be an oriented
bipartite graph on~$n$ vertices with $\delta^0(H)>n/k$. Then there are natural
numbers $a\neq b$ with $a+b<k$ such that~$H$ contains a
closed walk~$W$ of length~$a+b$ with~$a$ edges oriented
forwards and~$b$ edges oriented backwards.
\end{lem}

\begin{proof}Assume for a contradiction that~$H$ contains no
such walk~$W$. Let~$V_1$ and~$V_2$ be the colour classes
of~$H$ and without loss of generality  assume that $|V_1|\ge |V_2|$. 

Pick an arbitrary
vertex $x\in V_1$. 
Let $X_1:= N^+(x)$, $Y_1:=N^-(x)$, and for
each $i\ge 1$, set $X_{i+1}:=N^+(X_i)$ and
$Y_{i+1}:=N^-(Y_i)$. 
Now observe that $X_i\cap Y_j=\emptyset$ whenever
$i,j<\frac {k}{2}$ or both $i\neq j$ and $\max\{i,j\}=
\lceil\frac {k}{2}\rceil$. Indeed, this holds since our assumption on~$H$ implies that~$H$ contains no
directed cycle of length less than~$k$ and since
there is no cycle of length~$k$ when~$k$ is odd, as~$H$ is bipartite.


\begin{cl}
For all $i\neq j$ with $i, j\le\lceil\frac {k}2\rceil$, we have $X_i\cap
X_j=\emptyset$ and $Y_i\cap Y_j=\emptyset$.
\end{cl}

We prove only the first instance of the claim as the second is done
analogously. Suppose for a contradiction that $v\in X_i\cap
X_j$, for some  $j<i\le \lceil\frac {k}{2}\rceil$. By
definition of~$X_i$ there is a directed walk~$P$ of
length~$i$ from~$x$ to~$v$ and by
definition of~$X_j$ there is a directed walk~$P'$ of
length~$j$ from~$x$ to~$v$.
Then~$P\cup P'$
forms a closed walk of length $i+j\le 2\lceil\frac
{k}{2}\rceil -1$, with~$i$ edges going in one direction and~$j$ edges going in the other direction. If
$i+j=k$, then~$k$ is odd and this contradicts the assumption that~$H$ is
bipartite. If $i+j<k$ this contradicts our assumption that for all $a\neq b$ with $a+b<k$ there is no
closed walk~$W$ with~$a$ edges going forwards and~$b$
edges going backwards. This proves the claim.

Observe that for any $i\le \lceil\frac
{k}2\rceil$, we have $|X_i|, |Y_i|> \frac nk$. Consider first the case when~$k$
is odd. We obtain
\[
|V(H)|\ge \left|\left(\dot\bigcup_{i<\lceil\frac {k}2\rceil} (X_i\dot\cup
Y_i)\right)\dot\cup (X_{\lceil\frac {k}2\rceil}\cup Y_{\lceil\frac {k}2\rceil})\right|> \frac
{k-1}2\cdot \frac {2n}k+\frac nk=n,
\]
a contradiction. Now consider the case when~$4$ divides~$k$. Then we get
\[
|V_2|\ge \left|\dot\bigcup_{i\le\frac k4} (X_{2i-1}\dot\cup
Y_{2i-1})\right|> \frac
k4\cdot \frac {2n}k=\frac n2,
\]
a contradiction to the assumption that $|V_2|\le |V_1|$ and thus $|V_2|\le
\frac n2$. This finishes the proof of
Lemma~\ref{lem:bip-case}.
\end{proof}

Let us summarise the above observations in the following
lemma.
\begin{lem}\label{lem:resume}
Let~$k$ and~$\ell$ be as in $(*)$. Then any oriented graph~$H$ on $n\ge 64k$ vertices with
$\delta^0(H)>n/k$ contains a closed directed walk of length~$\ell$.
\end{lem}

\begin{proof}
Let~$k$,~$\ell$ and~$H$ be as in the lemma.  By
Proposition~\ref{prop:s-cycle} we may assume that~$H$ has
girth at least~$k$. This together with the assumption on
the minimum semidegree $\delta^0(H)>n/k$ and
Lemma~\ref{lem-diam} imply that~$H$ contains a
directed path of length at most~$64k$ between any ordered pair of vertices. 
Using Corollary~\ref{cor:or-cycle} we may assume that the
underlying graph of~$H$ contains no odd cycle of length less
than~$k$. Also observe that the underlying graph of~$H$ has
minimum degree greater than~$2n/k$. Now we can apply
Theorem~\ref{thm:AES} to the
underlying graph of~$H$ and deduce that~$H$ is bipartite. Recall that $(*)$ implies
that $k=p^s\ge 4$ for some prime~$p$ and some $s\in\mathbb{N}$. In particular,~$k$ is either odd or divisible by~$4$. 
Thus   Lemma~\ref{lem:bip-case}
implies that~$H$ contains a closed walk of length $a+b$
with~$a$ edges going forwards and~$b$ edges going backwards, for some
$a\neq b$ with $a+b<k$. Finally this together with Lemma~\ref{lem:or-cycles}
and our above observation that there must be a directed path
of length at most~$64k$ between any ordered pair of vertices of~$H$
imply that~$H$ contains a closed directed walk of length~$\ell$, as required.
\end{proof}

\begin{proof}[Proof of Theorem~\ref{thm:cycles}]

Apply Lemma~\ref{lem:KKO} with $c=\frac {1+\eta/2}{k}$,
$n_0=64k$ and $\varepsilon =\eta /2k$ to obtain  $n_1\in
\mathbb N$.

Let~$G$ be as in Theorem~\ref{thm:cycles}. Recall that by
 Lemma~\ref{lem:KKO}, in order to find an $\ell$-cycle in~$G$,
 it suffices to show that any oriented graph~$H$ on $n> 64k$ vertices and
 minimum semidegree $\delta^0(H)\ge (1+\eta/2)n/k$ contains a
 closed directed walk of length~$\ell$. This is done in
 Lemma~\ref{lem:resume}. 
\end{proof}
Finally, the following proposition shows that the condition that $n_0$ is large in Conjecture~\ref{conj:KDD} is necessary for $\ell>4$.
\begin{prop}\label{pro:counter}
Let $\ell> 4$ be an even integer and $k>2$ be minimal such that $k$ does not divide $\ell$. Then there exists a graph $G$ on $\lfloor\frac{k-1}2\rfloor(\ell-2)+1 $ vertices with $\delta^0(G)\ge \lfloor \frac nk\rfloor +1$ that does not contain any cycle of length greater than $\ell-1$.
\end{prop}
\begin{proof}
Let $G$ be the union of $\lfloor\frac{k-1}2\rfloor$ regular tournaments on $\ell-1$ vertices sharing a single vertex. Then $G$ does not contain any cycle of length greater than $\ell-1$, $\delta^0(G)=\frac {\ell-2}2$ and 
$n=\lfloor\frac{k-1}2\rfloor(\ell-2)+1$. Also
\[
\left\lfloor\frac nk\right\rfloor\le \left\lfloor\frac {\frac{k-1}{2}(\ell-2)+1}k\right\rfloor
= \left\lfloor\frac {\ell-2}2-\frac {\ell-4}{2k}\right\rfloor
\le  \frac {\ell-2}2-1=\delta^0(G)-1\;.
\]
\end{proof}
\section{Acknowledgement}
The authors would like to thank the referees for a careful reading of the manuscript and for pointing out an oversight in our proof.

  Daniela K\"uhn, Deryk Osthus, Diana Piguet\\
  School of Mathematics\\
  University of Birmingham\\
  Edgbaston\\
  Birmingham\\
  B15~2TT\\ 
  UK\\
  \\
  \texttt{e-mail addresses:\\
  \{d.kuhn, d.osthus, d.piguet\}@bham.ac.uk}
  
\end{document}